\documentclass[english,a4paper]{amsart}

\usepackage{babel}
\usepackage{amstext}
\usepackage{amsmath}
\usepackage{amsfonts}
\usepackage{amssymb}
\usepackage{latexsym}
\usepackage{ifthen}
\usepackage{hyperref}
\usepackage{xy}
\xyoption{all}

\newcommand\sE{{\mathcal E}}

\newcommand\sO{{\mathcal O}}

\newcommand\sC{{\mathcal C}}

\newcommand\bR{{\mathbb R}}

\newcommand\bP{{\mathbb P}}
\newtheorem{theorem}{Theorem}[section]
\newtheorem{lemma}[theorem]{Lemma}
\newtheorem{proposition}[theorem]{Proposition}
\newtheorem{corollary}[theorem]{Corollary}
\newtheorem{fact}[theorem]{Fact}
\newtheorem*{theorem*}{Theorem}
\theoremstyle{definition}
\newtheorem{remark}[theorem]{Remark}

\newtheorem{definition}[theorem]{Definition}
\newtheorem*{notation}{Notation}

\title[Bimeromorphic Geometry of Contact Threefolds]{On the Bimeromorphic Geometry\\ of Compact Complex Contact Threefolds} 
\author{Kristina Frantzen and Thomas Peternell}
\date{May 10, 2010}
\address{Mathematisches Institut, Universit\"at Bayreuth, 95440 Bayreuth, Germany\newline
\texttt{Kristina.Frantzen@uni-bayreuth.de, Thomas.Peternell@uni-bayreuth.de}}
\begin{document}
\begin{abstract}
We prove that a compact contact threefold which is bimeromorphically equivalent to a K\"ahler manifold and not rationally connected is the projectivised tangent bundle of a K\"ahler surface.
\end{abstract}
\maketitle
\tableofcontents
%
%
\section{Introduction}  
A (compact) complex manifold $X$ of dimension $2n+1$ is a {\it contact manifold} if there exists a vector bundle sequence 
\begin{equation}\label{contactseq}
0 \to F \to T_X \overset{\theta}{\to} L \to 0, \tag{$\star$}
\end{equation}
where $T_X$ is the tangent bundle and $L$ a line bundle, with the additional property that the induced map 
$\bigwedge^2 F \to L$,
$v \wedge w \mapsto [v,w]/ F \in L$,
given by the Lie bracket $[\, , \, ]$ on $T_X$ is everywhere non-degenerate. The line bundle $L$ is referred to as the \emph{contact line bundle} on $X$.

There are two basic ways to construct contact structures. 
\begin{itemize}
\item A simple Lie group gives rise to a 
Fano contact manifold $X$ (with $b_2(X) = 1$) by taking the unique closed orbit for the adjoint action of the Lie group on the projectivised
Lie algebra, see e.g. \cite{Be98}. 
\item For any compact complex manfold $Y$ the projectivised tangent bundle $X = \bP(T_Y)$ is a contact manifold. 
\end{itemize}
Now the question naturally arises
whether any compact complex contact manifold is given in this way. 

In the following, let $X$ be a compact complex contact manifold.
If $X$ is projective with $b_2(X) = 1$, then $X$ must be a Fano manifold and Beauville \cite{Be98} proved partial results towards the
realisation as closed orbit. 
In general, if $X$ is K\"ahler, Demailly \cite{De02} showed that the canonical bundle $K_X$ is not nef. 
If $X$ is projective with $b_2(X) \geq 2$, Theorem 1.1 in \cite{KPSW00} provides a positive answer to the question above. 
If $X$ is K\"ahler but not projective, then necessarily $b_2(X) \geq 2$ and the second alternative is conjectured to hold, i.e., $X$ should be
a projectivised tangent bundle. However the paper \cite{KPSW00} essentially uses Mori theory, which is, at the moment, not available in
the K\"ahler case, except in dimension 3 where it can be shown that $X$ is a projectivised tangent bundle over a surface (see Section \ref{SectionUniruled}).

In this paper we go one step further in dimension 3: we consider contact threefolds $X$ which are in  \emph{class $\sC$}, i.e., bimeromorphic
to a K\"ahler manifold. We first show that these threefolds must be uniruled. Then we consider the {\it rational 
quotient} $r: X \dasharrow Q$. The meromorphic map $r$ identifies two very general points if and only if they can be joined by a
chain of rational curves. In particular, $X$ is {\it rationally connected} if and only if $\dim Q = 0$. We distinguish the cases $\dim Q =1$ (Theorem \ref{main2}) and $\dim Q =2$ (Theorem \ref{main1}) and show
\begin{theorem*} Let $X$ be a compact contact threefold in class $\sC$. Assume that $X$ is not rationally connected. 
Then there exists a K\"ahler surface $Y$ such that $X \simeq \bP(T_Y)$. In particular, $X$ is K\"ahler.
\end{theorem*} 
The remaining open case that $X$ is rationally connected, in particular Moishezon, will require different methods. Probably
it will be necessary to consider rational curves $C$ with $-K_X \cdot C$ minimal, but positive. 
%
%
%
%
\section{Uniruledness and splitting}\label{SectionUniruled}
We shall use the following notation:
\begin{definition}
 A compact complex manifold $X$ is said to be in \emph{class $\sC$} if $X$ is bimeromorphically equivalent to a K\"ahler manifold.
\end{definition}
An important property of manifolds in class $\sC$ is the compactness of the irreducible components of the cycle space (cf.\ \cite{campana})

The key for our investigations is the following
\begin{theorem} \label{uniruled} 
Let $X$ be a compact contact threefold in class $\sC$. Then $X$ is uniruled.
\end{theorem} 
\begin{proof} 
Let 
$\pi: \hat X \to X$ 
be a modification such that $\hat X$ is K\"ahler.  It is a well-established fact that $\hat X$ is 
uniruled if and only if $K_{\hat X}$ is not pseudo-effective, i.e., the Chern class $c_1(K_{\hat X})$ is not represented by a positive closed 
current. The projective case in any dimension is treated in \cite{BDPP04} based on the uniruledness theorem of Miyaoka-Mori \cite{MM86}. The K\"ahler 
case in dimension three has been proved by Brunella \cite[Cor.\ 1.2]{Brunella}. 

The contact structure on $X$ is given by
$ \theta \in H^0(X,\Omega^1_X \otimes L)$; note that $\theta \wedge d\theta \neq 0$.
Via $\pi^*$, the $L$-valued form $\theta$ induces a form
$\hat \theta \in H^0(\hat X,\Omega^1_{\hat X} \otimes \pi^*(L))$.
By \cite[Cor.\ 1]{De02}, the pullback of the dual line bundle $\pi^*(L^{-1})$ is not pseudo-effective. Since $K_X = 2L^{-1}$, the line bundle
$\pi^*(K_X)$ is not pseudo-effective which is equivalent to say that $K_X$ is not pseudo-effective. 
Since $\pi_*(K_{\hat X}) = K_X$, the Chern class $c_1(K_{\hat X})$ cannot be represented by a positive closed current $\hat T$,
because otherwise $c_1(K_X)$ would be represented by the positive closed current $\pi_*(\hat T)$. 
Hence $K_{\hat X}$ is not pseudo-effective, and we conclude by the uniruledness criterion stated above.
\end{proof} 
As a consequence we obtain the following classification result for compact K\"ahler contact threefolds generalising the well-known projective case (see \cite{KPSW00} for further references). 
\begin{corollary} \label{coruni} 
Let $X$ be a compact K\"ahler contact threefold. Then either $X \simeq \bP_3$ or $X = \bP(T_Y)$ for a
K\"ahler surface $Y$. 
\end{corollary} 
\begin{proof} 
By Theorem \ref{uniruled} above, the threefold $X$ is uniruled. In particular, there is a positive-dimensional
subvariety through the general point of $X$, i.e., $X$ is not \emph{simple}.
The claim now follows from \cite[Theorem 4.1]{Pe01}.
\end{proof} 
\begin{remark} \label{classCimpliesPE}
The proof of Theorem \ref{uniruled} above actually shows that the canonical bundle of a compact contact manifold
in class $\sC$ of any dimension is not pseudo-effective. However in dimensions at least 4, unless $X$ is projective, it is completely open, whether this implies
uniruledness. 
\end{remark} 
%
%
%
We now make a digression and consider the contact sequence \eqref{contactseq}.
It is an interesting question whether this sequence can split. In the case where $X$ is Fano, LeBrun \cite[Cor.\,2.2]{LeBrun}
showed that splitting never occurs. By the following theorem, the same is true if $X$ is in class $\mathcal C$.
\begin{theorem} 
Let $X$ be a compact contact manifold in class $\mathcal C$.
Then the contact sequence \eqref{contactseq} does not split. 
\end{theorem} 
\begin{proof} 
Suppose we have a splitting $T_X \simeq F \oplus L$, hence $\Omega^1_X \simeq F^* \oplus L^*$. 
Then it is well-known (see e.g. Beauville \cite{Be00}), that 
$ c_1(L) \in H^1(X,L^*) \subset H^1(X,\Omega^1_X)$
and
$ c_1(F)  \in H^1(X,F^*)$.
Since $c_1(F) = n c_1(L) $ in $H^1(X,\Omega^1_X)$ and since $X$ is in class $\mathcal C,$ we conclude that
$ c_1(L) = c_1(F) = 0$
in $H^1(X,\Omega^1_X)$, and therefore also in $H^2(X,\bR)$. Hence $K_X$ is numerically trivial and 
consequently, due to Remark \ref{classCimpliesPE}, $X$ cannot be in class $\sC$.
\end{proof} 
%
%
%
Let $X$ be a compact contact manifold $X$ with contact sequence \eqref{contactseq}.
A subvariety $S \subset X$, i.e., closed irreducible analytic subset in $X$, is called \emph{$F$-integral} if $T_{S,x} \subset F_x$ for all smooth points $x \in S$. 
\begin{notation}
A \emph{holomorphic family $(C_t)_{t\in T}$ of curves in $X$} is given by a diagram
\[
 \begin{xymatrix}{
 Z \ar[r]^p \ar[d]^q & X\\
 T 
}
 \end{xymatrix}
\]
such that
\begin{itemize}
 \item 
$T$ is an irreducible subspace of the cycle space of curves in $X$
\item
$q^{-1}(t)$ is the cycle corresponding to $t \in T$
\item
$C_t = p (  q^{-1}(t)) $ as cycle. 
\end{itemize}
\end{notation}
An important tool will be the following lemma, the proof of which is based on the observation that a surface covered by a family of $F$-integral curves is itself $F$-integral. 
\begin{lemma} \label{integral} 
Let $X$ be a compact contact threefold with contact line bundle $L$. Let $(C_t)_{t\in T}$ be a 1-dimensional family 
of generically irreducible rational curves passing through a fixed point $x_0 \in X$. Then $L \cdot C_t \geq 2$. 
\end{lemma}  
\begin{proof}
We assume to the contrary that $L \cdot C_t \leq 1$ for all $t \in T$. By restricting the contact sequence \eqref{contactseq}
to an irreducible rational curve $C_t$ and, if necessary, pulling it back to the normalisation $\eta: \tilde{C_t} \to C_t$, one finds that the map $\mathcal O (2) \simeq T_{\tilde{C_t}} \to \eta^*L \simeq \mathcal O (a)$ with $a \leq 1$ is trivial and therefore $T_{C_t,x} \hookrightarrow F_x$  for $x \in C_t$ smooth. I.e., for the general $t \in T$ the curve $C_t$ is $F$-integral.

Consider the surface $S = \bigcup_{t \in T} C_t \subset X$ covered by the curves $C_t$. Then the proof of \cite[Proposition 4.1]{Keb01} shows that $S$ is $F$-integral. But since any $F$-integral
subvariety in $X$ has dimension at most 1, this yields a contradiction. 
\end{proof}
%
%
%
%
\section{The case of a 2-dimensional rational quotient}

We assume in this section that $X$ is a compact contact threefold in class $\sC$
with a rational quotient
$$ r: X \dasharrow Q $$
of dimension $\dim Q = 2$. We refer the reader to the books \cite{debarre}, \cite{Ko96} and the references therein for relevant details on the contruction and the properties of a rational quotient. A fundamental result of Graber, Harris, and Starr \cite{GHS} states that the quotient $Q$ is not uniruled. In case $X$ has dimension three, this result actually has previously been known. 

The meromorphic map $r: X \dasharrow Q $ is almost holomorphic, i.e., $r$ is proper holomorphic on a dense open set in $X$, 
and its general fiber is $\bP_1$. Thus,
we have a unique covering family $(l_t)_{t\in T}$ of rational curves with graph $Z$,
\[
 \begin{xymatrix}{
 Z \ar[r]^p \ar[d]^q & X \ar@{-->}[d]^r\\
 T  &  Q
}
 \end{xymatrix}
\]
and $\dim T = 2$. Since $r$ is almost holomorphic, the map $p$ is bimeromorphic.
By possibly passing to the normalisation, we may assume both $Z$ and $T$ normal. Moreover, we may take $Q = T$. 
\begin{lemma}  \label{irr}
All curves $l_t$ satisfy $L \cdot l_t = 1 $ and all irreducible curves $l_t$ are $F$-integral. 
\end{lemma}
\begin{proof} 
The general $l_t$ is a general fiber of $r$, hence $-K_X \cdot l_t = 2$ by adjunction and therefore $L \cdot l_t = 1$.
The same is then true for all $l_t$.  All irreducible curves $l_t$ are consequently $F$-integral (cf. proof of Lemma \ref{integral}). 
\end{proof} 
In the following we will make use of the deformation theory of rational curves. This is to say we consider a rational curve
$C \subset X$, given by a bimeromorphic morphism $f: \bP_1 \to X$ and consider the deformations $f_t$ of $f$.
We obtain a family 
$(C_t) = (f_t(\mathbb P_1))$ and then take its closure in the cycle space, because in general the family $(f_t)$ will not be compact, or in other
words, the family $(C_t)$ will split. Here it is essential that $X$ is in class $\sC$, hence all irreducible components of the cycle space of $X$ are compact.
We repeatedly use the following basic fact, see e.g. \cite[Theorem II.1.3]{Ko96}. 
\begin{fact}\label{DimOfDeform}
Let $X$ be a compact threefold and let $C$ be a rational curve in $X$.
If
$ -K_X \cdot C \geq m$,
then $C$ will deform as rational curve in an at least $m$-dimensional family.
\end{fact}
The following proposition is the technical core of this section.
\begin{proposition} \label{iso} 
The map $p: Z \to X$ is an isomorphism. In particular, the rational quotient $r: X \to T$ is holomorphic and equidimensional.
\end{proposition} 
\begin{proof} 
For $x \in X$ we let $T(x) $ be the analytic subset of all $t \in T$ such that $x  \in l_t$. 
Since the general $l_t$ does not pass through $x$, it follows that $\dim T(x) \leq 1$.
In the following we show that $\dim T(x) = 0$ for all $x$; in other words, $p$ is finite.
Since the map $p$ is of
degree $1$ and has connected fibers by Zariski's main theorem, the finiteness of $p$ forces $p$ to be biholomorphic. 

Suppose now to the contrary that $\dim T(x) = 1$ for some fixed $x \in X$. 
If $T(x)$ happens to be reducible, we replace it by an irreducible component of dimension 1. In the following, we shall therefore assume that $T(x)$ is irreducible.

Let $S$ be the surface covered by the $l_t$ belonging to $T(x)$:
\[
S = \bigcup_{t \in T(x)} l_t.
\]
If the general $l_t$ through $x$ is irreducible, then the Lemmata \ref{integral} and  \ref{irr} yield a contradiction.

So we are left with the case when all $l_t,\  t \in T(x)$ are reducible. In this case $S$ itself might be reducible.  
For $t \in T(x)$  we decompose $l_t$ into its irreducible components and write
$ l_t = \sum a_t^j C_t^j$. Since $L \cdot l_t = 1$ for all $t \in T$, there exists at least one component $C_t^j$ in this decomposition with $L \cdot C_t^j \geq 1$. We pick $t \in T(x)$ general and let $C^{(1)}$ be a component of $l_t$  with $L \cdot C^{(1)} \geq 1$. Then by Fact \ref{DimOfDeform}, $C^{(1)}$ deforms in an at least 2-dimensional family $(C_t^{(1)})_{t \in T_1}$. 

Suppose first that  
$$ L \cdot C^{(1)}_t = 1.$$ 
If the family $(C_t^{(1)})_{t \in T_1}$covers a surface, then we find a 1-dimensional subfamily through a
fixed point, contradicting Lemma \ref{integral}. If the family covers all of $X$,
then, since there is only
one covering family of generically irreducible rational curves in $X$, the family $(C^{(1)}_t)_{t \in T_1}$ must be the original family $(l_t)_{t\in T}$, in particular $T = T_1$. 
In other words, we have  $t_0 \in T(x) $ and $t_1  \in T $ such that
 $l_{t_0} = l_{t_1} + R$ 
with an effective curve $R$. Thus 
$p^{-1}(x) $ contains more than one point for every $x \in l_{t_1}$. Since $p$ has connected fibers, we conclude that 
$ \dim p^{-1}(x) = 1$
for every $x \in l_{t_1}$. Then either all curves $l_t,\ t \in T$  pass through $l_{t_1}$, which is absurd since $r$ is almost holomorphic, or there exists a 1-dimensional subfamily 
$(l_{t_1} + C_u)_{u \in U}$ of $(l_t)_{t\in T}$ with $\dim U = 1$. In this second case however, since the subfamily $(l_{t_1} + C_u)_{u \in U}$ does not contain the curve $l_{t_1}$ itself, it follows $p^{-1}(x) $ is not connected, a contradiction.

So we are left with 
$$L \cdot C^{(1)}_t \geq 2,$$ 
i.e., $ -K_X \cdot C^{(1)}_t \geq 4$, 
and the $C^{(1)}_t$ deform in an at least 4-dimensional family (cf. Fact \ref{DimOfDeform})
which must stay in an irreducible component 
$S_1$ of $S$. (The deformations of $C^{(1)}_t$ cannot cover all of $X$ since there is a unique covering family of generically irreducible rational curves in $X$, and this family, namely $(l_t)_{t\in T}$, is 2-dimensional and fulfils $- K_X \cdot l_t = 2$). 
We want to exhibit a new family $(C^{(2)}_t)$ in the surface $S_1$ such that $L \cdot C^{(2)}_t \geq 2$.  

In order to construct this new family we notice that the 4-dimensional family $(C^{(1)}_t)_{t \in T_1}$ must split. In fact, through any two points
of $S_1$ there is a positive-dimensional subfamily.
Now we choose carefully a splitting component $C^{(2)}$ such that $L \cdot C^{(2)} \geq 2$, namely we want to achieve that $C^{(2)} $ passes through a general point of $S_1$. 
By Lemma \ref{Moishezon} we obtain a generically non-splitting family $(h_u)_{u \in U}$ of rational curves $h_u$ in $S_1$ with $\dim U \geq 2$ such that for general $u \in U$ there 
exists $t(u) \in T_1$ such that $h_u$ is an irreducible component of $C^{(1)}_{t(u)}$. Since the family $(h_u)$ covers exactly $S_1$, there exists a 1-dimensional subfamily through a 
general point of $S_1$ and we take $C^{(2)}$ to be a general member of this subfamily. By Lemma \ref{integral} we obtain
$$ L \cdot C^{(2)} \geq 2.$$
Again, $-K_X \cdot C^{(2)} \geq 4$ implies that $C^{(2)}$ moves in an at least 4-dimensional family of rational curves, say $(C^{(2)}_t)_{t \in T_2}$. 

Inductively we obtain families $(C^{(k)}_t)_{t\in T_k}$ in $S_1$ such that 
$$ L \cdot C^{(k)}_t \geq 2.$$
Our aim now is to find an 
argument that this procedure must stop at some point, i.e., that $L \cdot C^{(k)}_t \geq 2$ cannot occur infinitely many times. 

If $X$ is K\"ahler with K\"ahler form $\omega$, this follows from the fact that the intersection number $C_t^{(k)} \cdot \omega$ strictly decreases and that
all classes $C_t^{(k)}$ are \emph{integer} classes in a ball 
$\{ a \, \vert\,  a \cdot \omega \leq K \} \subset H^2(X,\bR)$.

Let us briefly explain the difficulty arising from the fact that $X$ is not necessarily K\"ahler. 
If $X$ is merely in class $\sC$, we cannot argue in this way, because we will have curves with ``semi-negative'' cohomology. 
To be precise, we choose a sequence of blow-ups in points and smooth curves
$ \pi: \hat X \to X$
such that $\hat X$ is K\"ahler, fix a K\"ahler form $\hat \omega$ on $\hat X$, and form the current
$ R = \pi_*(\omega)$.
Then $R \cdot C > 0$ for all curves not contained in the center of $\pi$. On the other hand, there are finitely many curves $B_1, \ldots, B_N$
such that $ R \cdot B_j \leq 0$.
These ``bad'' curves have to be taken into account. 

We are able to get around this difficulty since every splitting takes place in
the fixed surface $S_1$. Inside this surface we will not have any curves with ``negative'' homology. 

We consider the normalisation
$$ \eta: \tilde S_1 \to S_1.$$
We define a family $(\tilde C^{(k)}_t)$ in $\tilde S_1$ by letting $\tilde C^{(k)}_t $ be the strict transform of $C^{(k)}_t$ in $\tilde S_1$ for general $t$ and then take
closure in the cycle space.
Let $\tilde C^{(k)}$ be the strict transform of $C^{(k)}$ in $\tilde S_1$. Then we obtain a splitting 
$$  \tilde C_{t_k}^{(k)} = \tilde C ^{(k+1)} + \tilde R _k $$
for some $t_k$. 
Inductively we find
\[
 \tilde C^{(m)} \equiv \tilde C_{t_m}^{(m)} = \tilde C^{(m+1)} + \tilde R_m,
\]
for all $m \in \mathbb N$. Here $\equiv$ denotes homology equivalence in $\tilde S_1$. It follows that 
\[
 \tilde C^{(1)} \equiv \tilde C^{(m)} + \sum_{j=1}^{m-1}\tilde R_j,
\]
i.e., $\tilde C^{(1)}$ is homology equivalent to a sum of arbitrary many effective curves in $\tilde S_1$. This contradicts Lemma \ref{split}.
\end{proof}
We now prove the two technical lemmata used in the proof of Proposition \ref{iso} above.
%
%
%
%
\begin{lemma} \label{split}
 Let $S$ be a compact connected normal Moishezon surface. Then there exists a linear map $\varphi: H_2(S, \mathbb Q) \to \mathbb Q$ such that $\varphi ([C]) \geq 1$  for all classes of irreducible curves $C$ in $S$.
\end{lemma}
\begin{proof}
It suffices to construct $\varphi$ on the subspace $V$ of $H_2(S; \mathbb Q)$ generated by classes of irreducible curves.
Let $\sigma: \hat S \to S$ be  a desingularisation of $S$ and note that the surface $\hat S$ is projective. We let $\hat H$ be an ample divisor on $\hat S$ and $\sigma_* ( \hat H) = H$ be its push-down to $S$. Using the intersection theory on normal surfaces established in \cite{MumfordIntersection} and \cite{Sakai}, we define 
\[
\varphi([C]) = C\cdot H = (\sigma^* C)\cdot (\sigma^* H).
\]
Here $\sigma^* D$ denotes the sum $\overline D + \sum a_i E_i$ of the strict transform $\overline D$ of the divisor $D$ in $\hat S$ and an appropriately weighted sum of the exceptional curves of $\sigma$. 

In order to check that $\varphi$ is well-defined on homology classes, it suffices to show that $c_1(\mathcal O(\sigma^*C))=0$ for every $C$ with $[C]=0 \in  H_2(S, \mathbb Q)$. 
Following the notation and results presented in  \cite{Sakai}, Section 3, this is equivalent to $c_1(\mathcal O (C))\in \mathrm{ker} (\sigma^*) = \mathrm{ker}(\eta_S) \subset H^2(S, \mathbb Q)$. Here
$\eta_S: H^2(S, \mathbb Q) \to H_2(S, \mathbb Q)$
denotes the Poincar{\'e} homomorphism on $S$. We may write
\[
\eta_S (c_1(\mathcal O (C))) = \sigma_*(\eta_{\hat S}(c_1( \widetilde{\mathcal O(C)})))
\]
 for the Poincar{\'e} isomorphism $\eta_{\hat S}: H^2(\hat S, \mathbb Q) \to H_2(\hat S, \mathbb Q)$ on $\hat S$ and $\widetilde{\mathcal O(C)} = \mathcal O (\sigma^*C)$. Since $\eta_{\hat S}(c_1(\mathcal O (\sigma^*C) )) = [\sigma^* C]$ on the smooth surface $S$, we conclude $\eta_S (c_1(\mathcal O (C))) = \sigma_*[\sigma^* C] = [C]=0$ and obtain the desired vanishing. 

It remains to check that $\varphi([C]) \geq 1$ for all classes of irreducible curves $C$ in $S$. We have 
\[
\varphi([C]) = (\sigma^* C)\cdot (\sigma^* H)  = (\sigma^* C) \cdot (\overline H + \sum b_i E_i).
\] 
Since $(\sigma^* C)\cdot E_j =0$ for all $j$ by definition of $\sigma^* C$ (cf. \cite[Section 1]{Sakai}), we conclude
\[
\varphi([C]) =(\sigma^* C)\cdot  \overline H =  (\overline C + \sum a_i E_i) \cdot \overline H.
\] 
Since $\overline H = \hat H $ is ample and $C$ is effective, in particular $\overline C$ is effective and $a_i >0$ for all $i$, the desired inequality follows. 
\end{proof}
\begin{lemma} \label{Moishezon} 
Let $S$ be an irreducible Moishezon surface with a covering family $(C_t)_{t \in T}$ of (rational) curves. Suppose $\dim T \geq 4$. 
Let $T' \subset T$ be the subset of those $t$ for which $C_t$ splits. Then $\dim T' \geq 2$.
\end{lemma} 

\begin{proof} Let $x \in S$ and
$T(x) = \{ t \in T \  \vert \ x \in C_t \}$.
Since $\dim T \geq 4$ by assumption, we have $\dim T(x) \geq 3$. (Consider the graph $p: Z \to S$ of the family $(C_t)_{t \in T}$ and observe that $\dim ( p^{-1}(x)) \geq 3$ and $q:Z \to T$ restricted to $p^{-1}(x)$ is finite.)
The same dimension count substituting $T$ by $T(x)$ shows that 
$$ \dim (T(x) \cap T(x')) \geq 2 $$
for $x, x' \in S$. Hence there exists a 2-dimensional subfamily through $x,x'$ and therefore we obtain a 1-dimensional subfamily through $x$ and $x'$ such that 
all members split. In other words
$$ \dim (T' \cap T(x) ) \geq 1. $$
Varying $x$ we conclude $\dim T' \geq 2$.
\end{proof} 
Having established Proposition \ref{iso}, it remains to show that the rational quotient $r: X \to T$ is actually a $\mathbb P_1$-bundle.
\begin{proposition} \label{bundle} Assume that the rational quotient $r: X \to T$ is holomorphic and equidimensional.
Then $r$ is a $\bP_1$-bundle, $T$ is smooth and $X = \bP(T_T)$.
\end{proposition}
\begin{proof} 
As a first step, we show that the fibers or $r$ must be irreducible. Assume the contrary and let $r^{-1}(t_0) = l_{t_0} = C^{(1)} + R$ be a reducible fiber 
such that $L \cdot C^{(1)} \geq 1$. Then $C^{(1)}$ deforms in an at least 2-dimensional family, hence $C^{(1)}$ is a member of $(l_t)_{t\in T}$, i.e.,
$C^{(1)} = l_{t_1} $ for a suitable $t_1 \in T$. Since $r$ is holomorphic, this is only possible when $t_0 = t_1$, a contradiction.

So $r: X \to T$ is a holomorphic, equidimensional map of normal complex spaces and every fiber of $r$ is a reduced, irreducible rational curve.  Now the arguments of \cite[Theorem II.2.8]{Ko96} can be adapted to our situation and it follows that $r$ is a $\bP_1$-bundle.
In particular, $T$ is smooth and $X = \bP(T_T)$.
\end{proof} 
Recall that a surface in class $\sC$ is K\"ahler. 
In total we have shown:
\begin{theorem} \label{main1}
Let $X$ be a compact contact threefold in class $\sC$. If the rational quotient has dimension 2, then $X$ is K\"ahler
and of the form $\bP(T_Y)$ with a K\"ahler surface $Y$. The projection $X \to Y$ is the rational quotient, i.e., $Y$ is not
uniruled. 
\end{theorem}   
%
%
%
%
\section{The case of a 1-dimensional rational quotient}
In this section we assume that $X$ is a compact contact threefold in class $\sC$ with contact line bundle $L$ 
and a rational quotient
$ r: X \dasharrow Q $
of dimension $\dim Q = 1$. 
Then necessarily $X$ is Moishezon and $Q$ is a smooth curve $B$ of genus at least $1$. We observe that
$r: X \to B$ is holomorphic. Our aim is to show that $X$ is of the form $X = \bP(T_Y)$ for some surface $Y$. The surface $Y$ is then necessarily Moishezon, and since
a smooth Moishezon surface is projective, we are going to show directly that $X$ is projective. 

Let $B_0$ be the set of points $b$ in $B$ such that $r^{-1}(b) = X_b$ is smooth.
\begin{lemma} \label{fiber} 
Let $b \in B_0$. Then $X_b$ is a Hirzebruch surface
$X_b = \bP(\sO_{\bP_1} \oplus \sO_{\bP_1}(-e))$
with $e > 0$ even.
\end{lemma}
\begin{proof} 
By adjunction $K_{X_b} = -2L \vert _ {X_b}$, hence $X_b$ is minimal. Moreover, $X_b$ cannot be a projective plane $\bP_2$  or a Hirzebruch 
surface with odd
$e$. To exclude the quadric ($e = 0$), observe that $X_b$ is not $F$-integral, i.e., the restriction of the contact form $\theta$ to $X_b$ does not vanish identically, and hence 
$$ \theta \vert_{X_b} \in H^0(X_b,\Omega^1_{X_b} \otimes L \vert_ {X_b}) \ne 0,$$ 
which is impossible for $X_b= \bP_1  \times \bP_1$. 
\end{proof} 
Every smooth fiber $X_b$ of $r$ has a uniquely defined non-splitting 1-dimensional family of rational curves, namely the ruling
lines. All these rational curves together give rise to a 2-dimensional family $(l_y)_{y \in Y}$ of rational curves in $X$, where
$Y$ is the irreducible component of the cycle space parametrising generically the ruling lines. 
We obtain an almost holomorphic map
$ \pi: X \dasharrow Y$ and a holomorphic map $g: Y \to B$, $g(y) = r(l_y)$,
\[
 \begin{xymatrix}{
X \ar@{-->}[r]^\pi & Y \supset Y_0 \ar[d]_g \\
  & B \supset B_0}  
 \end{xymatrix}
\]

such that $Y_0 = g^{-1}(B_0)$ is a $\bP_1$-bundle over $B_0$ and $\pi$ is again a $\bP_1$-bundle over $Y_0$. In fact,
if $Y_b$ is the fiber over $b \in B_0$, then
$X_b = \bP(\sO \oplus \sO(-e_b) \vert Y_b)$.

The technical key to the main result of this section is
\begin{proposition} \label{nonsplit}
The family $(l_y)_{y \in Y}$ does not split.
\end{proposition} 
\begin{proof}
Suppose $(l_y)$ splits. Then there exists a point $y_0 \in Y$, an irreducible curve $C^{(1)}$ with $L \cdot C^{(1)} \geq 1$
and an effective curve $R$ such that $l_{y_0} = C^{(1)} + R$. 
As seen before the curve $C^{(1)}$ deforms in an at least 2-dimensional family $(C^{(1)}_t)_{t \in T_1}$.

(1) \, Let us first consider the case $L \cdot C^{(1)}_t = 1$. If the family $(C^{(1)}_t)_{t \in T_1}$ covers a surface, we find a 1-dimensional subfamily through a fixed point and contradict Lemma \ref{integral}. If $(C^{(1)}_t)_{t \in T_1}$ covers all of $X$, we recover the original family $(l_y)$ as follows:  notice that the general $C_t^{(1)}$ will be an irreducible rational curve in a smooth fiber $X_b$ and $-K_{X_b} \cdot C_t = 2$ by adjunction. This implies that the general curve 
$C_t^{(1)}$ must be a ruling line, i.e., the general curve $C_t$ must be a curve $l_t$. This may now be excluded using Lemma \ref{integral} by the same arguments as in Proposition \ref{iso}.

(2) \, Having ruled out $L \cdot C^{(1)}_t = 1$, we consider the case $L \cdot C^{(1)}_t \geq 2$, i.e.,
$$ -K_X \cdot C^{(1)}_t = 2L \cdot C^{(1)}_t  \geq 4.$$ 
(2a) \, Assume that the family $(C^{(1)}_t)_{t \in T_1}$ covers all of $X$ and choose a general point $x \in X$. Dimension count shows that there is a 2-dimensional subfamily $(C^{(1)}_t)_{t \in T_1(x)}$ 
through the point $x$, necessarily
filling a rational surface $S$, which must be a fiber of $r$. Since $x$ is general, there is a $b \in B$ with $X_b$ smooth
such that $S = X_b$. The family $(C^{(1)}_t)_{t \in T_1(x)}$ splits as $C^{(1)}_{t_1} = C^{(2)} + R_1$ with $L \cdot C^{(2)} \geq 1$. If  $L \cdot C^{(2)} \geq 2$ we repeat to whole process. Assume that $L \cdot C^{(k)} \geq 2$ for all $k$. Then, we obtain a decomposition of the homology class of $C^{(1)}_{t_1}$ as a sum of arbitrary many effective curves $ C^{(1)}_{t_1} \equiv \sum_{k=1} ^K C^{(k)} + R_{K-1}$. 
As we can always choose a subfamily through a point of $S$, we can assume that $\sum_{k=1} ^K C^{(k)} + R_{K-1} \subset S$ for all $K$. Calculating the degree with respect to an ample line bundle $H$ on $S$, we obtain a contradiction.  Hence at some stage the procedure
has to stop, i.e. $L \cdot C^{(k)}_t = 1$ and we conclude by Lemma \ref{integral}. 

(2b) \, It remains to consider the  case the case where $L \cdot C^{(1)}_t \geq 2$ and the family $(C^{(1)}_t)$ covers a surface $S$, which is a component of a fiber $X_b$ of $r$. The family must split and we choose a splitting component $C^{(2)}$ such that $L \cdot C^{(2)} \geq 1$. If $L \cdot C^{(2)} = 1$, we are done again; if $L \cdot C^{(2)} \geq 2$, we obtain an at least 4-dimensional family  $(C^{(2)}_t)$. If this family covers $X$, we are done by the arguments of (2a) applied to $(C^{(2)}_t)$, instead of $(C^{(1)}_t)$. Otherwise, the family $(C^{(2)}_t)$ fills a component $S'$ of the same fiber $X_b$, and as in the proof of  Proposition \ref{iso}, by choosing $C^{(2)}$ carefully, we may assume that $S' = S$.  Now we are in completely the same situation as in the proof of Proposition \ref{iso} and proceed as described there.
\end{proof} 
In order to apply Proposition \ref{nonsplit}, we consider the normalised graph $p: Z \to X$ of the family $(l_y)_{y\in Y}$. 
\[
 \begin{xymatrix}{
 Z \ar[rr]^p \ar[dr]^q && X \ar@{-->}[dl]^\pi\\
 & Y& 
}
 \end{xymatrix}
 \]
\begin{proposition} 
The map $p: Z \to X$ is biholomorphic.
\end{proposition} 
\begin{proof} 
The map $p$ is generically biholomorphic,
hence by Zariski's main theorem, it suffices to show that $p$ does not have positive-dimensional fibers. 
So suppose that $\dim p^{-1}(x) = 1$. Then there exists a 1-dimensional subfamily $(l_y)_{y \in Y(x)}$ through $x$
with all $l_y$ irreducible by the previous proposition. Since $L \cdot l_y = 1$, we contradict Lemma \ref{integral}. 
\end{proof}
As before in Proposition \ref{bundle} we conclude:
\begin{corollary}
The map  $\pi: X \to Y$ is a $\bP_1$-bundle. 
\end{corollary}
We may now apply Lemma \ref{cont} below to the map $\pi: X \to Y$ and have shown:
\begin{theorem}\label{main2} 
Let $X$ be a compact contact threefold in class $\sC$ with 1-dimensional rational quotient $B$.
Then $X$ is projective and there is a smooth projective surface $Y$ with a $\bP_1$-fibration $Y \to B$ such 
that $X \simeq \bP(T_Y)$. 
\end{theorem}
\begin{lemma} \label{cont} 
Let $Y$ be a complex manifold of dimension $n+1$ and $\pi: X \to Y$ be a $\bP_n$-bundle. If $X$ is a contact manifold, then 
$X \simeq \bP(T_Y)$.
\end{lemma}
\begin{proof} 
Let $Z \simeq \mathbb P_n$ be a fiber of $\pi$. Adjunction implies that the contact line bundle $L$ restricted to $Z$ fulfils  $L|_Z = \mathcal O_Z(1)$.
Setting $ \sE = \pi_*(L)$, we conclude 
\[
(X,L) \simeq (\mathbb P ( \mathcal E), \mathcal O _{\mathbb P ( \mathcal E)} (1)).
\]
Now the last part of the proof of Theorem 2.12 in  \cite{KPSW00} can be applied 
and shows that $\sE \simeq T_Y$.
\end{proof} 
%

\begin{thebibliography}{KPSW00}

\bibitem[BDPP04]{BDPP04}
S{\'e}bastian Boucksom, Jean-Pierre Demailly, Mihai P{\u{a}}un, and Thomas
  Peternell, \emph{The pseudo-effective cone of a compact {K}\"ahler manifold
  of negative {K}odaira dimension}, arXiv:math/0405285, 2004.

\bibitem[Bea98]{Be98}
Arnaud Beauville, \emph{Fano contact manifolds and nilpotent orbits}, Comment.
  Math. Helv. \textbf{73} (1998), no.~4, 566--583.

\bibitem[Bea00]{Be00}
\bysame, \emph{Complex manifolds with split tangent bundle}, Complex analysis
  and algebraic geometry, volume in memory of Michael Schneider, eds. T.
  Peternell and F.O. Schreyer, de Gruyter, Berlin, 2000, pp.~61--70.

\bibitem[Bru06]{Brunella}
Marco Brunella, \emph{A positivity property for foliations on compact
  {K}\"ahler manifolds}, Internat. J. Math. \textbf{17} (2006), no.~1, 35--43.

\bibitem[Cam80]{campana}
Frederic Campana, \emph{Alg\'ebricit\'e et compacit\'e dans l'espace des cycles
  d'un espace analytique complexe}, Math. Ann. \textbf{251} (1980), no.~1,
  7--18.

\bibitem[Deb01]{debarre}
Olivier Debarre, \emph{Higher-dimensional algebraic geometry}, Universitext,
  Springer-Verlag, New York, 2001.

\bibitem[Dem02]{De02}
Jean-Pierre Demailly, \emph{On the {F}robenius integrability of certain
  holomorphic {$p$}-forms}, Complex geometry ({G}\"ottingen, 2000), volume in
  honour of H. Grauert, eds. I. Bauer et al., Springer, Berlin, 2002,
  pp.~93--98.

\bibitem[GHS03]{GHS}
Tom Graber, Joe Harris, and Jason Starr, \emph{Families of rationally connected
  varieties}, J. Amer. Math. Soc. \textbf{16} (2003), no.~1, 57--67.

\bibitem[Keb01]{Keb01}
Stefan Kebekus, \emph{Lines on contact manifolds}, J. Reine Angew. Math.
  \textbf{539} (2001), 167--177.

\bibitem[Kol96]{Ko96}
J{\'a}nos Koll{\'a}r, \emph{Rational curves on algebraic varieties}, Ergebnisse
  der Mathematik und ihrer Grenzgebiete. 3. Folge. A Series of Modern Surveys
  in Mathematics [Results in Mathematics and Related Areas. 3rd Series. A
  Series of Modern Surveys in Mathematics], vol.~32, Springer-Verlag, Berlin,
  1996.

\bibitem[KPSW00]{KPSW00}
Stefan Kebekus, Thomas Peternell, Andrew~J. Sommese, and Jaros{\l}aw~A.
  Wi{\'s}niewski, \emph{Projective contact manifolds}, Invent. Math.
  \textbf{142} (2000), no.~1, 1--15.

\bibitem[Le{B}95]{LeBrun}
Claude Le{B}run, \emph{Fano manifolds, contact structures, and quaternionic
  geometry}, Internat. J. Math. \textbf{6} (1995), no.~3, 419--437.

\bibitem[MM86]{MM86}
Yoichi Miyaoka and Shigefumi Mori, \emph{A numerical criterion for
  uniruledness}, Ann. of Math. (2) \textbf{124} (1986), no.~1, 65--69.

\bibitem[Mum61]{MumfordIntersection}
David Mumford, \emph{The topology of normal singularities of an algebraic
  surface and a criterion for simplicity}, Inst. Hautes \'Etudes Sci. Publ.
  Math. (1961), no.~9, 5--22.

\bibitem[Pet01]{Pe01}
Thomas Peternell, \emph{Towards a {M}ori theory on compact {K}\"ahler
  threefolds. {III}}, Bull. Soc. Math. France \textbf{129} (2001), no.~3,
  339--356.

\bibitem[Sak84]{Sakai}
Fumio Sakai, \emph{Weil divisors on normal surfaces}, Duke Math. J. \textbf{51}
  (1984), no.~4, 877--887.

\end{thebibliography}
%

\providecommand{\bysame}{\leavevmode\hbox to3em{\hrulefill}\thinspace}

\end{document}